\documentclass[12pt,oneside]{amsart}
\usepackage{amscd,amsmath,amssymb,amsfonts,a4}
\usepackage{hyperref}
\usepackage{color}

\newlength{\rulebreite}

\def\timesover#1#2#3{\ \xymatrix@1@=0pt@M=0pt{ _{#1}&\times&_{#2} \\& ^{#3}&}\ }
\def\otimesover#1#2#3{\ \xymatrix@1@=0pt@M=0pt{ _{#1}&\otimes&_{#2} \\& ^{#3}&}\ }

\usepackage[all]{xy}
\theoremstyle{plain}
\newtheorem{thm}{Theorem}

\newtheorem{cor}[thm]{Corollary}

\theoremstyle{definition}

\newtheorem{claim}[thm]{Claim}
\numberwithin{thm}{section}
\numberwithin{equation}{section}

\newcommand{\ga}[2]{
\begin{gather}\label{#1}#2\end{gather} 
}

\newcommand{\surj}{\twoheadrightarrow}

\newcommand{\Spec}{{\rm Spec \,}}

\newcommand{\sX}{{\mathcal X}}
\newcommand{\sY}{{\mathcal Y}}
\newcommand{\sZ}{{\mathcal Z}}

\newcommand{\N}{{\mathbb N}}

\newcommand{\Q}{{\mathbb Q}}

\newcommand{\Z}{{\mathbb Z}}

\begin{document}

\title[Point]{ Coniveau over $\frak{p}$-adic fields and points over finite fields}
\author{H\'el\`ene Esnault}
\address{
Universit\"at Duisburg-Essen, Mathematik, 45117 Essen, Germany}
\email{esnault@uni-due.de}
\date{April 6, 2007}
\thanks{Partially supported by  the DFG Leibniz Preis}

\parindent0cm
\parskip5pt

\begin{abstract}
If the $\ell$-adic cohomology of a projective smooth variety,
defined over a $\frak{p}$-adic field $K$ with finite residue field $k$, is supported in codimension $\ge 1$, 
then any model over the ring of integers of $K$ has a $k$-rational point. 
\end{abstract}
\maketitle
{\bf Version fran\c{c}aise abr\'eg\'ee.}
Soit $X$ une vari\'et\'e projective et absolument irr\'eductible sur un corps local $K$.  Rappelons qu'un {\it mod\`ele} de $X/K$ sur l'anneau de valuation $R$ de $K$ est un morphisme 
$\sX\to \Spec R$ projectif et plat, tel que $(\sX\to \Spec R)\otimes K=(X\to \Spec K)$.  Nous consid\'erons la cohomologie $\ell$-adique $H^i(\bar X)$ \`a coefficients dans $\Q_\ell$.  Dire qu'elle est support\'ee en codimension 1 signifie que toute classe dans $H^i(\bar X)$ a une restriction nulle dans $H^i(\bar U)$, o\`u $U\subset X$ est un ouvert non vide. 
Le but de cette note est de prouver le th\'eor\`eme suivant.

{\bf Th\'eor\`eme:} {\it Soit} $X$ {\it une vari\'et\'e  projective lisse et absolument irr\'eductible sur un corps local} $K$ {\it de caract\'eristique 0 et \`a corps r\'esiduel fini} $k$. {\it On suppose que  la cohomologie} $\ell$-{\it adique} $H^i(\bar X)$ {\it est support\'ee en codimension} $\ge 1$
 {\it pour tout} $i\ge 1$.
 {\it Soit}  $\sX/R$ {\it un mod\`ele. } 
 {\it Alors il existe un morphisme projectif surjectif}  $\sigma: \sY\to \sX$ {\it de} $R$-{\it sch\'emas tel que} $|\sY(k)| \equiv 1$ {\it modulo} $|k|$. 

On en d\'eduit imm\'ediatement le corollaire suivant.\\
{\bf Corollaire:} {\it Sous les hypoth\`eses du th\'eor\`eme, tout mod\`ele} $\sX/R$ {\it poss\`ede  un point} $k$-{\it rationnel}.

Pour ce qui concerne l'existence du point $k$-rationnel, ceci affranchit  \cite[Theorem~1.1]{Ept}, (qui est vrai aussi si $K$ est de  caract\'eristique $p>0$),  de l'hypoth\`ese de r\'egularit\'e sur le choix du mod\`ele  $\sX$, qui \'etait utilis\'ee pour pouvoir appliquer le th\'eor\`eme de puret\'e de Gabber  \cite{Fu}. Pour ce faire, nous montrons que d'avoir des singularit\'es quotient est suffisant, de m\^eme que pour l'\'etude de l'application de sp\'ecialisation. Nous appliquons alors la version plus pr\'ecise du th\'eor\`eme de Jong ainsi qu'elle est expos\'ee dans 
\cite{deJ2}. 
\section{Introduction}
Let $X$ be a projective, absolutely irreducible  variety defined over a local field $K$ with finite residue field $k$. 
Recall that a {\it model} of $X/K$ on the valuation ring $R$ of $K$ is a flat projective morphism $\sX\to \Spec R$ such that 
$(\sX\to \Spec R)\otimes K=(X\to \Spec K)$. We  consider $\ell$-adic cohomology $H^i(\bar X)$ with $\Q_\ell$-coefficents. One defines the first coniveau level 
\ga{N}{N^1H^i(\bar X)=\{\alpha \in H^i(\bar X), \exists \ {\rm divisor} \ D\subset X \ {\rm s.t.} \ 
0=\alpha|_{X\setminus D}\in H^i(\overline{X\setminus D})\}.}
As $H^i(\bar X)$ is a  finite dimensional $\Q_\ell$-vector space, one has by localization
\ga{N2}{\exists D \subset X \ {\rm s.t.} \ N^1H^i(\bar X)= {\rm Im}\big(H^i_{\bar D}(\bar X)\to H^i\bar X)\big),}
where $D\subset X$ is a divisor. One says that $H^i\bar X)$ is {\it supported in codimension 1} if $N^1H^i(\bar X)=H^i(\bar X)$. This definition is general, but has good properties only if $X$ is irreducible and smooth or has only very mild singularities.

In \cite[Theorem~1.1]{Ept} it is shown that if $X/K$ is smooth, projective, absolutely irreducible over a local field $K$ with finite residue field $k$, and if $\ell$-adic cohomology $H^i(\bar X)$
is supported in codimension $\ge 1$ for all $i\ge 1$, then any regular model $\sX/R$ of $X/K$  has the property
\ga{1.1}{\sX(k)\equiv 1 \ {\rm mod} \ |k|.}

The purpose of this note is to drop the regularity assumption if $K$ has characteristic 0.

\begin{thm} \label{thm1.1}

Let $X$ be a smooth,  projective, absolutely irreducible variety defined over a local field $K$  of characteristic 0 with finite residue field $k$. Assume that $\ell$-adic cohomology $H^i(\bar X)$
is supported in codimension $\ge 1$ for all $i\ge 1$. Let $\sX$ be a model of $X$ over the ring of integers $R$ of $K$. Then  there is a 
projective surjective morphism $\sigma: \sY\to \sX$ of $R$-schemes  such that 
 $$|\sY(k)|\equiv 1 \ {\rm mod} \ |k|.$$ 
\end{thm}
As an  immediate corollary, one obtains
\begin{cor} \label{cor1.2}
 Under the assumptions of the theorem, every model $\sX/R$ has a $k$-rational point. 
\end{cor}

The regularity of the model $\sX$ in the proof of \cite[Theorem~1.1]{Ept} (which is shown also when $K$ has characteristic $p>0$) was used to apply Gabber's purity theorem \cite{Fu}. We show that for the piece of regularity one needs, it is enough to have quotient singularities. Likewise, for the properties needed on the specialization map, quotient singularities are good enough. The more careful use of de Jong's  theorem as exposed in \cite{deJ2} allows then to conclude. 

{\it Acknowledegment:} This note relies on de Jong's fundamental alteration theorems. T. Saito suggested to us the use of them in the shape formulated in \cite{deJ2}. We thank him for this, and for many subsequent discussions on the subject. 
We exposed a weaker version of Theorem \ref{thm1.1}  at the conference in honor of S. Bloch in Toronto in March 2007. Discussions with him, A. Beilinson and L. Illusie contributed to simplify our original exposition. 
\section{Proof of Theorem \ref{thm1.1}}
Let $K$ be a local field of characteristc 0 with finite residue field $k$. Let $R\subset K$ be its valuation ring. 
Let $\sX\to \Spec R$ be an integral model of a projective variety $X\to \Spec K$. We do not assume here that $X$ is absolutely irreducible, nor do we assume that $X/K$ is smooth. Then by \cite[Corollary~5.15]{deJ2}, there is a diagram
\ga{2.1}{\xymatrix{\ar[drr] \sZ\ar[r]^{\pi} & \ar[dr] \sY \ar[r]^{\sigma}   & \sX \ar[d]\\
  & & \Spec R
}
}
and a finite group $G$ acting on $\sZ$ over $\sY$
with the properties
\begin{itemize}
\item[(i)] $\sZ\to \Spec R$ and $\sY\to \Spec R$ are  flat,
\item[(ii)] $\sigma$ is projective, surjective, and birational,
\item[(iii)] $\sY$ is the quotient of $\sZ$ by $G$,
\item[(iv)] $\sZ$ is regular. 
\end{itemize}
So $\sY\to \Spec R$ is not quite a model of $\sX\to \Spec K$, but is close to it. We show in the sequel that $\sigma$ in \eqref{2.1} does it in Theorem \ref{thm1.1}. 
Set
$$Y=\sY\otimes K, \ Z=\sZ\otimes K.$$

For an open $U\subset X$ let us set $Y_U=U\times_X Y, \ Z_U=U\times_X Z$. 

Let us assume now that $X/K$ is smooth. 
This implies that 
\ga{2.2}{ H^i(\bar U)\xrightarrow{\sigma^* \ {\rm inj}}  H^i(\overline{Y_U}).}
Moreover,
one 
has a trace map from $Y$ to $X$ 
\ga{2.4}{
\xymatrix{ H^i(\overline{Y_U})  \ar[rr]^{(\rm trace)_{Y/X}} & & H^i(\bar U)
}
}
which splits $\sigma^*$ in \eqref{2.2}. Let $i\ge 1$ and 
let $D\subset X$ be a divisor such that $H^i_{\bar D}(\bar X)\surj H^i(\bar X)$ and such that 
$\sigma|_{X\setminus D}: Y\setminus \sigma^{-1}(D)\to X\setminus D$ is an
 isomorphism.
Then \eqref{2.4} yields the commutative diagram
\ga{2.5}{\xymatrix{\ar[d]_{({\rm trace})_{Y/X}}H^i(\bar{Y}) \ar[r] &  
H^i(\overline{Y\setminus \sigma^{-1}(D)})\ar[d]^{=}\\
H^i(\bar X) \ar[r]^{0} & H^i(\overline{X\setminus D})
}
}
and we conclude
\ga{2.6}{X/K \ {\rm smooth}  \Longrightarrow N^1H^i(\bar X)=H^i(\bar X)\subset N^1H^i(\overline{Y})=H^i(\overline{Y}).}

We endow all schemes considered (which are   $R$-schemes) with the upper subscript $^u$ to indicate the base change $\otimes_R R^u$ or $\otimes_K K^u$, where $K^u\supset K$ is the maximal unramified extension, and $R^u\supset R$ is the normalization of $R$ in $K^u$. Likewise, we write $\overline{?}$
to indicate the base change $\otimes_R \bar R, \ \otimes_K \bar K, \ \otimes_k \bar k$, where $\bar K\supset K, \ \bar k\supset k$ are the algebraic closures and $\bar R\supset R$ is the normalization of $R$ in $\bar K$. 
We consider as in \cite[(2.1)]{Ept} the $F$-equivariant exact sequence (\cite[3.6(6)]{DeWeII})
\ga{2.7}{\ldots \to H^i_{\bar B}(\sY^u)\xrightarrow{\iota} H^i(\bar B)=H^i(\sY^u)\xrightarrow{sp^u} H^i(Y^u) \to \ldots, 
}
where  $F \in {\rm Gal}(\bar k/k)$ is the geometric Frobenius, and $B=\sY\otimes k$.

One has
\begin{claim} \label{claim2.1}
 The eigenvalues of the geometric Frobenius $F\in {\rm Gal}(\bar k/k)$ acting on $H^i(X^u)$ and on $H^i( Y^u)$ lie in $q\cdot \bar{\Z}$ for all $i\ge 1$. 
\end{claim}
\begin{proof}
For $H^i(X^u)$, this is \cite[Theorem~1.5(ii)]{Ept}. 
One has $H^i(\bar Y)=H^i(\bar Z)^G$, thus in particular, $\pi^*: H^i(\bar  Y)\to H^i(\bar Z)$ is injective. By \eqref{2.6} one has
\ga{2.3}{H^i(\bar Y)\xrightarrow{\pi^* \ {\rm inj}} N^1H^i(\bar Z).}
Since $K$ has characteristic 0, and $\sZ$ is regular by (iv), $Z$ is smooth.
Thus we can apply again \cite[Theorem~1.5(ii)]{Ept}. This finishes the proof.

\end{proof}
\begin{claim}\label{claim2.2}
 The eigenvalues of the geometric Frobenius $F\in {\rm Gal}(\bar k/k)$  acting on
$\iota(H^i_{\bar B}(\sY^u))\subset H^i(\bar B)$ lie in $q\cdot \bar{\Z}$ for all $i\ge 1$.
\end{claim}
\begin{proof}
 By (iii), one has $H^i_{\bar B}(\sY^u)=H^i_{\bar C}(\sZ^u)^G\subset H^i_{\bar C}(\sZ^u)$, where $C=\pi^{-1}(B)$. Since by (iv), $\sZ$ is regular, we can apply \cite[Theorem~1,4]{Ept}, which is a consequence of Gabber's purity theorem \cite{Fu}, to conlude. 

\end{proof}

\begin{proof}[Proof of Theorem \ref{thm1.1}]
Claims \ref{claim2.1} and \ref{claim2.2} together with \eqref{2.7} show that the eigenvalues of $F$ acting on $H^i(\bar{B})$ lie in $q\cdot \bar{\Z}$ for all $ i\ge 1$.

We apply the Lefschetz trace formula $|B(k)|={\rm Tr}F|H^*(\bar B)$. As $B$ is absolutely connected and defined over $k$, $F|H^0(\bar B)={\rm Identity}$.  By the discussion, one has  $|B(k)| \in \N\cap (1+q\cdot \bar{\Z})\subset 1+q\cdot \Z$.  
\end{proof}
\section{Remarks}
Starting from Theorem \ref{thm1.1}, and Corollary \ref{cor1.2}, we may ask what happens if $K$ has equal characteristic $p>0$ and whether or not the congruence of the theorem is true on all models. We have no counter-examples for either question. What $K$ is concerned, characteristic 0 is used in 
the proof of Claim \ref{claim2.1}: if $K$ has characteristic $p>0$, we only know that $Z$ is regular, thus we can't apply immediately \cite[Theorem~1.5(ii)]{Ept}. Going up to a strict semi-stable model does not help as for this, one has to ramify $R$ and one loses regularity of $\sZ$ and $Z$. What the congruence is concerned, instead of going to one birational model $\sY$ (or birational up to some inseprable extension in characteristic $p>0$), one should go up to a hypercover built out of such $\sY$. In doing Deligne's construction of hypercovers with resolutions of singularities being replaced by de Jong's morphisms of the type $\sigma$ in \eqref{2.1}, one creates components which do not dominate $\sX$, the cohomology of which is very hard to control. So one perhaps loses the coniveau property.

\bibliographystyle{plain}
\renewcommand\refname{References}

\end{document}